\documentclass[12pt]{article}
\usepackage[T1]{fontenc}
\usepackage{lmodern}
\usepackage[utf8]{inputenc}
\usepackage{libertinus}
\usepackage[margin=1in]{geometry}
\usepackage{amsmath,amssymb,amsthm,mathtools}
\usepackage{booktabs}
\usepackage{enumitem}
\usepackage{microtype}
\usepackage[colorlinks=true,linkcolor=blue,citecolor=blue,urlcolor=blue]{hyperref}
\theoremstyle{definition}
\usepackage[nameinlink,noabbrev]{cleveref}
\newtheorem{theorem}{Theorem}
\newtheorem{conjecture}{Conjecture}
\newtheorem{definition}{Definition}
\newtheorem{lemma}{Lemma}

\newtheorem{question}{Question}
\newtheorem{claim}{Claim}
\newtheorem{subclaim}{Claim}[claim]

\newtheorem{proposition}{Proposition}

\newcommand{\ex}{{\rm ex}}
\newcommand{\EX}{{\rm EX}}

\newcommand{\spex}{{\rm spex}}
\newcommand{\SPEX}{{\rm SPEX}}

\newcommand{\wh}{\widehat}
\newcommand{\cM}{\mathcal{M}}
\newcommand{\cS}{\mathcal{S}}

\newcommand{\floor}[1]{\left\lfloor #1 \right\rfloor}
\newcommand{\ceil}[1]{\left\lceil #1 \right\rceil} 
\begin{document}
	\title{Spectral Tur\'an problems for suspensions of balanced trees}
	\author{Yaoxiang Di
			\thanks{School of Mathematical Sciences, Anhui University,	Hefei  230601, P.~R.~China. Email:{\tt dyxlgkk@stu.ahu.edu.cn}. Supported in part by the National Natural Science Foundation of China (No.12471319) and Anhui Provincial Department of Education Research Project (No. 2025AHGXZK50063).}
			\and Chunyang Dou \thanks{School of Mathematical Sciences, Anhui University, Hefei  230601, P.~R.~China. Email:{\tt chunyang@stu.ahu.edu.cn}. Supported in part by the National Natural Science Foundation of China (No.12471319) and Anhui Provincial Department of Education Research Project (No. 2025AHGXZK50154).
			}
}
	\date{}
	\maketitle
\begin{abstract}
A central problem in spectral Tur\'an theory is to understand the relationship between the spectral extremal family $\SPEX(n,F)$ and the ordinary extremal family $\EX(n,F)$. For many forbidden graphs $F$, it is known that $\SPEX(n,F)\subseteq\EX(n,F)$ holds for infinitely many $n$, while only a few examples have been identified where the two families are disjoint. In this paper, we study this problem for suspensions of balanced trees.  A tree is balanced if its two bipartition classes differ in size by at most one. Let $T$ be a balanced tree on $2k$ or $2k+1$ vertices and  $\widehat T$ be its suspension which is  obtained from $T$ by adding one new vertex adjacent to every vertex of $T$. Our first main result establishes  a tight  upper bound for  the spectral Tur\'an number of $\widehat T$ for sufficiently large $n$ provided that $T$ satisfies some mild assumptions. Our second result determines for which integers $k$ and which non-path balanced trees $T$ on $2k$ or $2k+1$ vertices there are infinitely many integers $n$ such that $\EX(n,\widehat{T})\cap \SPEX(n,\widehat{T})=\emptyset$.

\end{abstract}

\section{Introduction}
Given a graph $F$, a graph $G$ is called $F$-free if it contains no copy of $F$ as a subgraph.  The \emph{Tur\'an number} $\ex(n,F)$ is the maximum number of edges in an $n$-vertex $F$-free graph.  An $n$-vertex $F$-free graph with $\ex(n,F)$ edges is called an extremal graph for $\ex(n,F)$, and the family of all such graphs is denoted by $\EX(n,F)$.

For a graph $G$, let $\rho(G)$ denote the largest eigenvalue of its adjacency matrix, also called the \emph{spectral radius} of $G$.  In analogy with the Tur\'an number, the \emph{spectral Tur\'an number} $\spex(n,F)$ is the maximum value of $\rho(G)$ over all $n$-vertex $F$-free graphs $G$.  An $n$-vertex
$F$-free graph with spectral radius $\spex(n,F)$ is called an extremal graph for $\spex(n,F)$, and the family of all such graphs is denoted by $\SPEX(n,F)$.  Since the average degree of a graph is at most its spectral radius, we have
\[
\ex(n,F)\le \tfrac{n}{2}\spex(n,F).
\]
Motivated by this inequality, spectral Tur\'an problems have been studied as a way to obtain upper bounds for  ordinary Tur\'an problems.  The systematic study of $\spex(n,F)$ was initiated by Nikiforov \cite{N10}, following several earlier sporadic results.  Nikiforov \cite{N02} also obtained a celebrated theorem for $\spex(n,K_{r+1})$, which implies the value of $\ex(n,K_{r+1})$ and hence recovers Tur\'an's theorem \cite{T41}. Let the friendship graph $F_k$ be  the graph obtained by adding a new vertex adjacent to every vertex of the matching $M_k$ of size $k$. More recently, Cioab\u{a}, Feng, Tait, and Zhang \cite{CFTZ20} studied the maximum spectral radius of graphs without the friendship graph $F_k$ for sufficiently large $n$ and proved that $\SPEX(n,F_k)\subseteq\EX(n,F_k)$ for infinitely many integers $n$.  This result prompted substantial further work. For more results on when $\SPEX(n,F)\subseteq\EX(n,F)$ for infinitely many integers $n$, we refer the reader to \cite{D22,FTZ24,LP22,NWK23, WNKF,WKX23}.

However, the spectral radius and the edge count measure different features of a graph.  A simple example illustrates the distinction.  Let $m$ be large and choose $n=\binom{m}{2}+1$. The star $K_{1,n-1}$ and $H = K_m \cup (n-m)K_1$ have the same number of edges, but $\rho(K_{1,n-1})=\sqrt{n-1}$ while $\rho(H)=m-1\ge\sqrt{2n}-1$. Thus the spectral radius can distinguish graphs with the same edge count when their edges are organized very differently. It is therefore natural to ask not only when ordinary and spectral extremal graphs coincide, but also when their extremal families are disjoint:
\[
\SPEX(n,F) \cap \EX(n,F) = \emptyset
\]
for infinitely many integers $n$.

We now review the related results.  For a graph $F$, the \emph{suspension} $\widehat F$ is the graph obtained by adding a new vertex adjacent to every vertex of $F$.  For example, the wheel $W_{\ell}$ is the suspension of a cycle $C_{\ell-1}$ on $\ell-1$ vertices, and the kipas graph $\widehat P_\ell$ is the suspension of a path $P_\ell$ on $\ell$ vertices. Yuan \cite{Y21} determined the Tur\'an number of $W_{2k+1}$-free graphs and characterized all ordinary extremal graphs for sufficiently large $n$.  Later, Cioab\u{a}, Desai, and Tait \cite{CDT22} determined the corresponding spectral Tur\'an number and described all spectral extremal graphs for sufficiently large $n$.  Combining these results, for every $k\ge 9$, there exist infinitely many integers $n$ such that $\EX(n,W_{2k+1})\cap\SPEX(n,W_{2k+1})=\emptyset$.  Yuan \cite{Y21} also studied the Tur\'an number of the kipas graph $\widehat P_\ell$ for sufficiently large $n$, and the ordinary extremal graphs can be deduced from his proof.  Subsequently, Zhang \cite{Z} determined the corresponding spectral Tur\'an number.  Consequently, for every $k\ge 9$, there exist infinitely many integers $n$ such that $\EX(n,\widehat P_{2k})\cap\SPEX(n,\widehat P_{2k})=\emptyset$ and $\EX(n,\widehat P_{2k+1})\cap\SPEX(n,\widehat P_{2k+1})=\emptyset$.

Since a path is a special tree, it is natural to consider the suspension of a more general tree. A tree on a single vertex is called a \emph{trivial} tree. We shall use the following famous conjecture of Erd\H{o}s and S\'os.
\begin{conjecture}
For any  nontrivial tree $T_q$ on $q$ vertices, $\ex(n,T_q)\le \frac{q-2}{2}n$.
\end{conjecture}
Although the Erd\H{o}s-S\'os conjecture remains open in general, some special cases have been established, including trees of diameter at most four \cite{M05} and trees in which a vertex is adjacent to at least $\frac{|V(T)|-1}{2}$ leaves \cite{S89}. A proof for very large trees was announced in 2008 by Ajtai, Koml\'os, Simonovits, and Szemer\'edi.

A tree $T$ is called balanced if the two bipartition classes differ in size by at most one. Let $T$ be a balanced tree on $2k$ or $2k+1$ vertices, and let $\widehat T$ be  its suspension.  Zhu, Wang, Zhang, and Zhang \cite{Zhu} studied $\ex(n,\widehat{T})$ and proved the following theorem.
\begin{theorem}[Zhu, Wang, Zhang, and Zhang \cite{Zhu}]\label{classical}
Let $T$ be a balanced tree on $2k$ or $2k + 1$ vertices, and suppose the Erd\H{o}s--S\'os conjecture holds for all of its nontrivial subtrees.  When
$n \geq 4(4k)^6$, we have
\[
\ex(n, \widehat{T}) \leq \max\left\{ n_0 n_1 + \left\lfloor \frac{(k - 1)n_0}{2} \right\rfloor : n_0 + n_1 = n \right\}.
\]

Moreover,
\begin{enumerate}
    \item If $2(k - 1) \mid \left\lfloor \dfrac{2n + k}{4} \right\rfloor$, then equality holds.
    \item If $T$ is on $2k + 1$ vertices and $2 \mid \left\lfloor \dfrac{2n + k}{4} \right\rfloor$, then equality holds.
    \item If $T$ contains a matching of size $k$, then equality holds for every $n$.
\end{enumerate}
\end{theorem}
In view of the results of Yuan \cite{Y21} and Zhang \cite{Z} for $\widehat{P}_\ell$, it is natural to ask whether the same disjointness phenomenon occurs for suspensions of balanced trees.
\begin{question} 
For which integers $k$ and which non-path balanced trees $T$ on $2k$ or $2k+1$ vertices are there infinitely many integers $n$ such that 
$\EX(n,\widehat{T})\cap \SPEX(n,\widehat{T})=\emptyset$? Does this occur for infinitely many values of $k$?
\end{question}
Motivated by this question, we study the spectral Tur\'an problem for the suspension of a balanced tree on $2k$ or $2k+1$ vertices.  We also answer the question affirmatively in \Cref{disjoint}.  To state our results, we first introduce the following notation.
We define
\[
\phi(n,k)=\frac{k-1+\sqrt{(k-1)^2+4\lfloor n/2\rfloor\lceil n/2\rceil}}{2},
\qquad
\psi(n,k)=\frac{k-1+\sqrt{(k-1)^2+n^2-4}}{2}.
\]
For $n\equiv2\pmod4$, let $\eta(n)$ be the largest root of $x^3-x^2-\frac{n^2}{4}x+\frac{n}{2}=0$.

To describe the extremal families, we need the decomposition family introduced by Simonovits~\cite{S}. For a graph $H$ with $\chi(H)=r+1$, the \emph{decomposition family} $\mathcal{M}(H)$ consists of the minimal graphs $M$ such that embedding a copy of $M$ into one partite class of the Tur\'an graph $T_r(n)$ with $n$ large enough forces a copy of $H$ (the precise definition is given in \Cref{decomp}). Here $T_r(n)$ is the complete $r$-partite graph on $n$ vertices with partite sizes as equal as possible.

If $n\not\equiv 2 \pmod{4}$ or $k$ is odd, then for $a\in\{\floor{n/2},\ceil{n/2}\}$, let $\mathcal{G}_{a,k-1}$ be the family of graphs obtained from the complete bipartite graph with partite classes of sizes $a$ and $n-a$ by embedding a $(k-1)$-regular $\mathcal{M}(\widehat{T})$-free graph into the partite class of size $a$.  For even $k\ge4$ and $n\equiv2\pmod4$, let $\mathcal{H}_{n,k-1}$ be the family of graphs obtained from
$K_{n/2-1,\,n/2+1}$ by embedding a $(k-1)$-regular $\mathcal{M}(\widehat{T})$-free graph into one of its partite classes.  Finally, when $n\equiv2\pmod4$ and $k=2$, let $\mathcal{J}_{n,2}$ be the family of graphs obtained from $K_{n/2,\,n/2}$ by embedding a maximum matching into one of its partite classes.
Let
\[
\tau(n,k) =
\begin{cases}
\eta(n) & n\equiv 2 \pmod{4} \text{ and } k=2,\\
\psi(n,k) &  n\equiv 2 \pmod{4}\text{ and } \text{even } k\ge 4,\\
\phi(n,k) & \text{otherwise}.
\end{cases}
\]
Similarly, let
\[
\mathcal{K}(n,k)=
\begin{cases}
\mathcal{J}_{n,2} & n\equiv 2 \pmod{4} \text{ and } k=2,\\
\mathcal{H}_{n,k-1} &  n\equiv 2 \pmod{4}\text{ and } \text{even } k\ge 4,\\
\mathcal{G}_{\floor{n/2},k-1}\cup\mathcal{G}_{\ceil{n/2},k-1} & \text{otherwise}.
\end{cases}
\]
The following is one of the main theorems.
\begin{theorem}\label{main}
Let $T$ be a balanced tree on $2k$ or $2k+1$ vertices, and assume that the
Erd\H{o}s--S\'os conjecture holds for every nontrivial subtree of $T$.  Then
\[
\spex(n,\widehat{T})\le \tau(n,k)
\]
for sufficiently large $n$. Equality holds only if  $\SPEX(n,\widehat{T})\subseteq \mathcal{K}(n,k)$.
Moreover, 
      \begin{enumerate}
      \item For $k=2$,  there are infinitely many
$n\equiv2\pmod4$ for which the bound $\eta(n)$ is attained.

\item   For every fixed even $k\ge4$, there are infinitely many $n\equiv2\pmod4$ for which the
bound $\psi(n,k)$ is attained.
          \item For every fixed $k\ge2$, there are infinitely many
$n\not\equiv2\pmod4$ for which the bound $\phi(n,k)$ is attained.

      \end{enumerate}                   
\end{theorem}
We conclude the introduction with our disjointness result.
\begin{theorem}\label{disjoint}
The following hold.
\begin{enumerate}
	\item For infinitely many integers $k$, every balanced tree $T$ on $2k+1$ vertices that is not a path and satisfies the Erd\H{o}s--S\'os conjecture for every nontrivial subtree of $T$ has infinitely many integers $n$ such that $\EX(n,\widehat{T})\cap\SPEX(n,\widehat{T})=\emptyset$.
	\item For infinitely many integers $k$, every balanced tree $T$ on $2k$ vertices that is not a path, contains a matching of size $k$, and satisfies the Erd\H{o}s--S\'os conjecture for every nontrivial subtree of $T$ has infinitely many integers $n$ such that $\EX(n,\widehat{T})\cap\SPEX(n,\widehat{T})=\emptyset$.
\end{enumerate}
\end{theorem}

Since the proofs of our main results rely on several independent techniques, we present them in a self‑contained manner.  The rest of the paper is organized as follows.  In Section~2, we collect the necessary definitions and auxiliary lemmas.  Section~3 contains the complete proofs of Theorems~\ref{main} and \ref{disjoint}. 

\section{Preliminaries}
We begin by collecting the combinatorial tools that will allow us to bound the number of edges inside each part of the partition obtained from the stability lemma.
        
Given two graphs $L$ and $R$, the \emph{join} $L\vee R$ is the graph obtained from the vertex-disjoint union of $L$ and $R$ by connecting each vertex of $L$ and each vertex of $R$.
		
\begin{definition}[Simonovits \cite{S}]\label{decomp}
Given a graph $H$ with $\chi(H) = r+1$, the \emph{decomposition family} $\mathcal{M}(H)$ is the set of minimal graphs $M$ such that $H \subset (M \cup \overline{K_t}) \vee T_{r-1}\left((r-1)t\right)$, where $t = t(H)$ is a constant.
\end{definition}

The next observation gives a simple way to guarantee that certain constructions are $\widehat T$-free.
\begin{proposition}\label{prop}
If a graph $F$ contains no member of $\cM(H)$ as a subgraph, then embedding $F$ into one partite class of any complete $r$-partite  graph on $n$ vertices yields an $H$-free graph, where $\chi(H)=r+1$.
\end{proposition}	

To apply this observation, we need the decomposition family of $\widehat T$, which was determined by Zhu, Wang, Zhang, and Zhang \cite{Zhu}.
		
\begin{lemma}[Zhu, Wang, Zhang, and Zhang \cite{Zhu}]\label{decomp sus tree}
Let $T$ be a balanced tree whose smaller color class has size $k$.  A set $S\subseteq V(T)$ is called a covering if $V(T)\setminus S$ is independent.
Define
\[
\cS(T):=\{T[S]: S\subseteq V(T)\text{ is a covering and }|S|\le k\} \quad\text{and}\quad \widehat{\cS(T)}=\{\wh{T[S]}:T[S]\in \cS(T)\}.
\]
Then $\cM(\wh T) \subseteq \{T\}\cup \widehat{\cS(T)}$.  Moreover, the star $K_{1,k}$ is in $\widehat{\cS(T)}$ and every other member of  $\widehat{\cS(T)}$
contains a triangle. 
\end{lemma}
		
The next result converts an edge surplus beyond the Erdős--Sós bound into a collection of vertex‑disjoint copies of prescribed trees.
\begin{lemma}[Zhu, Wang, Zhang, and Zhang \cite{Zhu}]\label{tree-packing}
Let $K\ge k'\ge0$, and let $G$ be an $n$-vertex graph with $n\ge 7K^3$, $\Delta(G)\le K-1$, and $e(G)\ge \frac{k'n}{2}+3K^3$. Let $T_1,\dots,T_\omega$ be trees with $\omega\le 2K$ and $|T_i|\le k'+2$ for all $i$. If the Erd\H{o}s--S\'os bound $\ex(m,T_i)\le \frac{(|T_i|-2)m}{2}$ holds for every $m\in\mathbb{N}$ and for each $i$ with $|T_i|\ge 2$, then $G$ contains vertex-disjoint copies of $T_1,\dots,T_\omega$.
\end{lemma}
		
In order to use the  lemma above, we decompose $T$ by removing a small independent set.
\begin{lemma} [Zhu, Wang, Zhang, and Zhang \cite{Zhu}]\label{splitting}
Let $T$ be a balanced tree on $2k$ or $2k+1$ vertices.  Then, for every $a\in\{1,\dots,k\}$, there exists an independent set $I_a\subseteq V(T)$ with $|I_a|\le a$ such that
\[
T-I_a=T_1\cup\cdots\cup T_{\omega},
\]
where $\omega\le 2k$ and $|T_i|\le k-a+1$ for every $i$.
\end{lemma}

The proof relies on a spectral stability partition for vertex‑critical graphs. Recall that a graph $F$ is \emph{vertex-critical} if $\chi(F-v)=\chi(F)-1$ for some $v\in V(F)$ and $\chi(F)\ge3$. Since $\wh T$ is vertex‑critical, the following result of Zhang~\cite{Z} applies.
\begin{lemma}[Zhang \cite{Z}] \label{stability}
Let $T$ be a balanced tree and let $\theta>0$.  Then there exists $n_0=n_0(T,\theta)$ such that the following holds for every $n\ge n_0$.
		
If $G$ is an $n$-vertex $\wh T$-free graph with $\rho(G)=\spex(n,\wh T)$, then $G$ is connected and there is a partition $V(G)=V_0\cup V_1$ such that
\[
\left||V_i|-\frac n2\right|<\theta n \qquad (i=0,1),
\]
and, for every $v\in V_i$,
\[
d_{V_i}(v)<|V(\wh T)|,\qquad d_G(v)\ge n-|V_i|.
\]
Moreover, if $x=(x_v)_{v\in V(G)}$ is a Perron vector of $G$ normalized by $\max_{v\in V(G)}x_v=1$, then
\[
x_v>1-\theta \qquad \text{for all }v\in V(G).
\]
\end{lemma}
		
We now record two standard spectral facts.  Let $G=(V,E)$ be a simple graph of order $n$ with adjacency matrix $A$.
\begin{definition}[Godsil and Royle \cite{GR01}]
 A partition $\pi=\{V_1,\dots,V_m\}$ of $V$ is \emph{equitable} if every  $v\in V_i$ has exactly $d_{ij}$ neighbors in $V_j$.  The matrix $Q=(d_{ij})_{i,j=1}^{m}$ is the quotient matrix of $G$ with respect to $\pi$.
\end{definition}

Equitable partitions reduce several spectral computations below to small quotient matrices.  The irreducibility condition identifies the Perron root.
\begin{definition}[Zhan \cite{Zhan}]
A nonnegative square matrix $M$ is \emph{irreducible} if its associated digraph $\mathcal{D}(M)$ is strongly connected, where $i\to j$ whenever $M_{ij}>0$.
\end{definition}

The following standard lemma is invoked whenever an equitable quotient matrix is used to compute a spectral radius.
\begin{lemma}[Godsil and Royle \cite{GR01}]\label{equitable partition}
Let $G$ have adjacency matrix $A$, and let $\pi=\{V_1,\dots,V_m\}$ be an
equitable partition with quotient matrix $Q$.  If $\lambda$ is an eigenvalue of $Q$ with eigenvector $\mathbf{x}=(x_1,\dots,x_m)^\mathsf{T}$, then the vector $\tilde{\mathbf{x}}$ defined by $(\tilde{\mathbf{x}})_v=x_i$ for all $v\in V_i$  is an eigenvector of $A$ for $\lambda$.  If $G$ is connected and $Q$ is irreducible,
then $\rho(G)=\rho(Q)$.
\end{lemma}

We also need an upper bound for joins in terms of the maximum degrees inside the two parts.
\begin{lemma} [Tait \cite{MT19}]\label{join spectral bound}
Let $H_0$ and $H_1$ be graphs on $n_0$ and $n_1$ vertices, respectively, with maximum degrees $d_0$ and $d_1$.  Let
\[
H:=H_0\vee H_1
\]
be their join.  Then
\[
\rho(H)\le
\frac{d_0+d_1+\sqrt{(d_0-d_1)^2+4n_0n_1}}{2}.
\]
Furthermore, equality can hold only if both $H_0$ and $H_1$ are regular.
\end{lemma}

The following elementary estimate is used to find large common neighbors.
\begin{lemma}[Cioab\u{a}, Feng, Tait, and Zhang \cite{CFTZ20}]\label{inclusion exclusion lemma}
If $A_1, \ldots, A_q$ are finite sets, then 
\[|A_1 \cap \ldots \cap A_q| \geq \sum_{i=1}^q |A_i| - (q-1)\bigg| \bigcup_{i=1}^q A_i \bigg|.\]
\end{lemma}
\section{Proof of Theorems \ref{main} and \ref{disjoint}}
\begin{proof}[Proof of  \Cref{main}] To prove the upper bound, it suffices to consider a $\wh T$-free graph $G$ on $n$ vertices satisfying $\rho(G)=\spex(n,\widehat{T})$ for sufficiently large $n$. Let $x=(x_v)_{v\in V(G)}$ be a Perron vector normalized by $\max_{v\in V(G)}x_v=1$.
	
Set
\[
\theta:=\frac{1}{400k}.
\]
By \Cref{stability}, since $n$ is sufficiently large, there is a partition
\[
V(G)=V_0\cup V_1
\]
such that
\begin{equation}\label{stability-balanced-suspension}
\left||V_i|-\frac n2\right|<\theta n,\qquad d_{V_i}(v)<|V(\wh T)|,\qquad d_G(v)\ge n-|V_i|,\qquad x_v>1-\theta
\end{equation}
for every $i\in\{0,1\}$ and every $v\in V_i$. Thus
\begin{equation}\label{cross degree}
    d_{V_{1-i}}(v)=d_G(v)-d_{V_i}(v)\ge |V_{1-i}|-d_{V_i}(v)
\end{equation}
for every $v\in V_i$.
	
\vspace{3pt}
\begin{claim}\label{max degree}
For $i\in\{0,1\}$, we have
\begin{equation}
\Delta(G[V_i])\le k-1.
\end{equation}
\end{claim} 
\begin{proof}[Proof of Claim \ref{max degree}]
Suppose not. Some $u\in V_i$ has $k$ neighbors $u_1,\dots,u_k$ in $V_i$. By \eqref{stability-balanced-suspension} and \eqref{cross degree}, $d_{V_{1-i}}(v)\ge |V_{1-i}|-2k-1$ for all $v\in V_i$, thus \Cref{inclusion exclusion lemma} yields for $n\ge 50k^3$,
\[
|N_{V_{1-i}}(u)\cap\cdots\cap N_{V_{1-i}}(u_k)| \ge (k+1)(|V_{1-i}|-2k-1)-k|V_{1-i}| \ge (\frac12-\theta)n-(2k+1)(k+1)\ge k+1.
\]
The bipartite graph between $\{u_1,\dots,u_k\}$ and this common neighbor contains $K_{k,k+1}$, hence a copy of $T$ and together with $u$ this yields $\wh T$, a contradiction.
\end{proof}	
If $k=1$, then Claim \ref{max degree} implies that both $G[V_0]$ and $G[V_1]$ are edgeless. Thus $G$ is bipartite, and 
\[
\rho(G)\le \sqrt{|V_0||V_1|} \le \sqrt{\lfloor n/2\rfloor\lceil n/2\rceil} =\phi(n,1).
\]
	Equality holds here for the complete bipartite graph
	$K_{\lfloor n/2\rfloor,\lceil n/2\rceil}\in \mathcal{G}_{\floor{n/2},k-1}\cup\mathcal{G}_{\ceil{n/2},k-1}$. Thus we may assume from now on that $k\ge 2$.
	
Let
\[
d_i:=\Delta(G[V_i])\qquad (i=0,1).
\]
\begin{claim}\label{sum degree}
We have 
\begin{equation}
	d_0+d_1\le k-1.
\end{equation}
\end{claim}
\begin{proof}[Proof of Claim \ref{sum degree}]
Suppose that $d_0+d_1\ge k$. After relabeling the parts if necessary, we may assume that $d_1\ge 1$ and $d_0\ge k-d_1$ by the assumption $d_0+d_1\ge k$.
	
Choose a vertex $y\in V_0$ with $d_{V_0}(y)=d_0$, and choose distinct neighbors $y_1,\dots,y_{k-d_1}\in N_{V_0}(y)$. Define
\[
Y:=N_{V_1}(y)\cap N_{V_1}(y_1)\cap\cdots\cap N_{V_1}(y_{k-d_1}).
\]
By \eqref{stability-balanced-suspension} and \eqref{cross degree}, each of the $k-d_1+1$ vertices $y,y_1,\dots,y_{k-d_1}$ has at most $|V(\wh T)|-1=|V(T)|\le 2k+1$ non-neighbors in $V_1$. Therefore,
\begin{equation}\label{eq:Y-large}
|V_1|-k(2k+1)\le |Y|\le |V_1|.
\end{equation}
	
Since
\[
|V_1|-k(2k+1) \ge \left(\frac12-\theta\right)n-k(2k+1)\ge 7k^3
\]
for $n\ge 50k^3$, we have $|Y|\ge 7k^3$. Apply \Cref{splitting} with $a=k-d_1$. We obtain an independent set $I\subseteq V(T)$ with $|I|\le k-d_1$ such that
\[
T-I=T_1\cup\cdots\cup T_R,
\]
where $R\le 2k$ and $|T_j|\le d_1+1$ for every $j$.  Let $\mathcal{T}^{+}=\{T_j:|T_j|\ge 2\}$. Since each $T_j\in \mathcal{T}^{+}$ is a subtree of $T$, the Erd\H{o}s--S\'os conjecture holds for $T_j\in \mathcal{T}^{+}$ by hypothesis.
	
Now $G[Y]$ has maximum degree at most $d_1$. 
\begin{subclaim}\label{eV1}
\begin{equation}
e(G[V_1])< \frac{(d_1-1)|V_1|}{2}+5k^3.
\end{equation}
\end{subclaim}
\begin{proof}[Proof of Claim \ref{eV1}]
We first show that $e(G[Y])<\frac{(d_1-1)|Y|}{2}+3k^3$. If
\[
e(G[Y])\ge \frac{(d_1-1)|Y|}{2}+3k^3,
\]
then \Cref{tree-packing}, applied with $K=k$ and $k'=d_1-1$, yields vertex-disjoint copies of $T_1,\dots,T_R$ in $G[Y]$. Since every vertex of $Y$ is adjacent to every vertex among $y,y_1,\dots,y_{k-d_1}$, the graph $G[N(y)]$ contains a copy of $T$, and hence $G$ contains a copy of $\wh T$, a contradiction. Therefore,
\begin{equation}\label{Y-edge-upper}
e(G[Y])<\frac{(d_1-1)|Y|}{2}+3k^3.
\end{equation}
Every edge of $G[V_1]$ either lies inside $Y$ or is incident with a vertex of $V_1\setminus Y$. Hence 
\[
e(G[V_1])\le e(G[Y])+(|V_1|-|Y|)d_1.
\]
Using \eqref{eq:Y-large} and \eqref{Y-edge-upper}, and recalling that $|V_1|-|Y|\le k(2k+1)$ and $d_1\le k-1$, we obtain that
\[
e(G[V_1])\le e(G[Y])+(|V_1|-|Y|)d_1< \frac{(d_1-1)|Y|}{2}+3k^3+2k^3\le \frac{(d_1-1)|V_1|}{2}+5k^3.
\]
\end{proof}
Choose $z\in V_1$ with $d_{V_1}(z)=d_1$, and let $z_1,\dots,z_{d_1}$ be its neighbors in $V_1$.  Put
\[
X:=N_{V_0}(z)\cap N_{V_0}(z_1)\cap\cdots\cap N_{V_0}(z_{d_1}).
\]
As before, $|V_0|-|X|\le k(2k+1)$ and $|X|\ge 7k^3$.  Applying \Cref{splitting} with $a=d_1$, we obtain an independent set $I'\subseteq V(T)$ with $|I'|\le d_1$ such that
\[
T-I'=S_1\cup\cdots\cup S_M,
\]
where $M\le 2k$ and $|S_j|\le k-d_1+1$ for all $j$.
\begin{subclaim}\label{ev0}
\begin{equation}
e(G[V_0])< \frac{(k-d_1-1)|V_0|}{2}+5k^3.
\end{equation}
\end{subclaim}
\begin{proof}[Proof of Claim \ref{ev0}] We first show that $e(G[X])< \frac{(k-d_1-1)|X|}{2}+3k^3$.
If 
\[
e(G[X])\ge \frac{(k-d_1-1)|X|}{2}+3k^3,
\]
then \Cref{tree-packing}, applied with $K=k$ and $k'=k-d_1-1$, yields vertex-disjoint copies of $S_1,\dots,S_M$ in $G[X]$. Since every vertex of $X$ is adjacent to every vertex among $z,z_1,\dots,z_{d_1}$, the graph $G[N(z)]$ contains a copy of $T$, and hence $G$ contains a copy of $\wh T$, a contradiction. Therefore,
\[
e(G[X])<\frac{(k-d_1-1)|X|}{2}+3k^3.
\]
Arguing exactly as above in Claim~\ref{eV1}, and using $|V_0|-|X|\le k(2k+1)$ and $d_0\le k-1$, we obtain
\[
e(G[V_0])< \frac{(k-d_1-1)|V_0|}{2}+5k^3.
\] 
\end{proof}
From Claims \ref{eV1} and \ref{ev0}, and the trivial bound $e(V_0,V_1)\le |V_0||V_1|$, we obtain
\begin{align}
e(G)&< |V_0||V_1|+\frac{(k-d_1-1)|V_0|}{2}+\frac{(d_1-1)|V_1|}{2}+10k^3\nonumber\\
    &= |V_0||V_1|+\frac{k-1}{2}|V_0|-\frac12|V_1|+\frac{d_1}{2}\bigl(|V_1|-|V_0|\bigr)+10k^3\nonumber\\
	& \le|V_0||V_1|+\frac{k-1}{2}|V_0|-\frac12|V_1|+k\theta n+10k^3.
\label{total-bound}
\end{align}
	
Now, let $F_0$ be the disjoint union of as many copies of $K_{k-1,k-1}$ as possible on the vertex set $V_0$, together with the remaining isolated vertices. Then $\Delta(F_0)\le k-1$, and because fewer than $2k-2$ vertices remain uncovered,
\begin{equation} \label{F0-edge-lower-bound}
e(F_0)\ge \frac{k-1}{2}|V_0|-(k-1)^2.
\end{equation}
Since  $F_0$ is triangle-free with maximum degree at most $k-1$, and every component of $F_0$ has order at most $2k-2<|V(T)|$,  these facts implies that $F_0$ contains neither $T$ nor any graph of the form $\wh{T[S]}$ from \Cref{decomp sus tree}. Consequently, $F_0$ contains no member of $\cM(\wh T)$ as a subgraph.
	
Let $H$ be obtained from the complete bipartite graph with partite classes $V_0$ and $V_1$ by embedding $F_0$ into $V_0$. By \Cref{prop} and the previous
paragraph, $H$ is $\wh T$-free. Combining \eqref{total-bound} and \eqref{F0-edge-lower-bound} with $|V_1|\ge n/2-\theta n$, we get
\begin{align}
e(H)-e(G)&\ge \frac12|V_1|-k\theta n-10k^3-(k-1)^2\nonumber\\
		&\ge \left(\frac14-\left(k+\frac12\right)\theta\right)n-11k^3.
\label{edge-gap-HG}
\end{align}
Because $\theta=1/(400k)$, the coefficient of $n$ in \eqref{edge-gap-HG} is positive. More precisely, we obtain 
\begin{equation}\label{eq:edge-gap-linear}
e(H)-e(G)\ge \left(\frac14-\frac{k+\frac12}{400k}\right)n-11k^3>0.246\,n-11k^3.
\end{equation}

Let
\[
\alpha:=|E(H)\setminus E(G)|,\qquad \beta:=|E(G)\setminus E(H)|.
\]
 Since each vertex of $V_0$ has at most $|V(\wh T)|-1=|V(T)|\le 2k+1$ non-neighbors in $V_1$, the number of missing cross-edges of $G$ is at most
\begin{equation}
		(2k+1)|V_0|<(2k+1)\left(\frac12+\theta\right)n,
\end{equation}
while
\[
e(F_0)\le \frac{k-1}{2}|V_0|<\frac{k-1}{2}\left(\frac12+\theta\right)n.
\]
Thus
\begin{equation}\label{eq:A-upper}
\alpha\le(2k+1)|V_0|+e(F_0)<(2k+1)\left(\frac12+\theta\right)n+\frac{k-1}{2}\left(\frac12+\theta\right)n<\frac{5}{2}kn.
\end{equation}
Also, by \eqref{eq:edge-gap-linear},
\begin{equation}\label{eq:AminusB-lower}
\alpha-\beta>0.246\,n-11k^3.
\end{equation}
Since every Perron entry satisfies $x_v>1-\theta$, we have
\begin{align}
x^\top(A(H)-A(G))x&=2\sum_{uv\in E(H)\setminus E(G)}x_ux_v-2\sum_{uv\in E(G)\setminus E(H)}x_ux_v\\
&\ge 2(1-\theta)^2\alpha-2\beta =2(\alpha-\beta)-2(2\theta-\theta^2)\alpha.
\end{align}
Using \eqref{eq:A-upper}, we have
\[
2(2\theta-\theta^2)\alpha<4\theta \alpha<\frac{n}{50}.
\]
Combining this with \eqref{eq:AminusB-lower}, we see that for every $n\ge 50k^3$, the quantity $2(\alpha-\beta)-2(2\theta-\theta^2)\alpha>0$. Hence
\[
\rho(H) \ge \frac{x^\top A(H)x}{x^\top x}>\frac{x^\top A(G)x}{x^\top x}=\rho(G),
\]
contradicting the extremality of $G$. This contradiction proves Claim \ref{sum degree}.
\end{proof}
    We shall use the standard fact that adding an edge to a connected graph strictly increases its spectral radius.

\noindent\textbf{Case 1:} Assume $n\not\equiv2\pmod4$ or $k$ is odd.  Let $G^\ast:=G[V_0]\vee G[V_1]$.  By \Cref{join spectral bound} and Claim \ref{sum degree}, 
{\small{\[
	\rho(G)\le \rho(G^\ast)\le
	\frac{d_0+d_1+\sqrt{(d_0-d_1)^2+4|V_0||V_1|}}{2}
	\le \frac{k-1+\sqrt{(k-1)^2+4\lfloor n/2\rfloor\lceil n/2\rceil}}{2}\le \phi(n,k).
	\]}}
The second-to-last inequality holds as $|V_0||V_1|\le \lfloor n/2\rfloor\lceil n/2\rceil,d_0+d_1\le k-1$ and $|d_0-d_1|\le d_0+d_1\le k-1$. If equality holds, then $G=G^\ast$, $|V_0||V_1|=\lfloor n/2\rfloor\lceil n/2\rceil$, $d_0+d_1=k-1$, and $|d_0-d_1|=k-1$.  Thus one partite class is independent and, by the equality condition in \Cref{join spectral bound}, the other partite class induces a $(k-1)$-regular graph.  Since $G$ is $\wh T$-free, the induced $(k-1)$-regular graph is $\cM(\wh T)$-free. Hence $G\in \mathcal{G}_{\floor{n/2},k-1}\cup\mathcal{G}_{\ceil{n/2},k-1}$.

\noindent\textbf{Case 2:} Let $n\equiv2\pmod4$ and $k=2$. After relabeling, we may assume that $\Delta(G[V_0])\le1$ and that $G[V_1]$ is independent.  Write $a=|V_0|$ and $b=|V_1|$.  Let $F_{a,b}$ be obtained from $K_{a,b}$ by adding a maximum matching in the partite class of size $a$.  Then $G\subseteq F_{a,b}$,  and hence $\rho(G)\le \rho(F_{a,b})$.

Suppose first that $a$ is even. Then the matching is perfect. The partition of $F_{a,b}$ into the two partite classes is equitable with quotient matrix
\[
\begin{pmatrix}
1&b\\ 
a&0
\end{pmatrix}.
\]
By \Cref{equitable partition}, $\rho(F_{a,b})$ is the larger root of $x^2-x-ab=0$.  Since $\tfrac{n}{2}$ is odd and $a$ even, we have $ab\le (n/2)^2-1$, which implies that  $\rho(F_{a,b})\le\lambda_0$, where $\lambda_0$ is the largest root of
\[
x^2-x-\left(\frac{n}{2}\right)^2+1=0.
\]
Next, we claim that $\lambda_0<\eta(n)$. Let $h_{n/2}(x):=x^3-x^2-(\tfrac{n}{2})^2x+\frac{n}{2}$.  Recall that $\eta(n)$ is the largest root of $h_{n/2}(x)=0$. Since $h_{n/2}(x)$ is a cubic function in one variable with a positive leading coefficient, to prove the claim above, we only need to show $\tfrac{n}{2}<\lambda_0$, $h_{n/2}(\lambda_0)<0$, and $h_{n/2}(x)$ is strictly increasing on $[\tfrac{n}{2},\infty)$. Let $g(x):=x^2-x-\left(\frac{n}{2}\right)^2+1$. Since $g(n/2)=-\tfrac{n}{2}+1<0$, it implies that $\tfrac{n}{2}<\lambda_0$. Indeed, we have $h_{n/2}(\lambda_0)=\tfrac{n}{2}-\lambda_0<0$.   It remains to  prove that $h_{n/2}(x)$ is strictly increasing on $[\tfrac{n}{2},\infty)$. Since  
\[
h'_{n/2}(x)=3x^2-2x-(\frac{n}{2})^2\ge 3(\frac{n}{2})^2-2(\frac{n}{2})-(\frac{n}{2})^2=\frac{n^2}{2}-n>0
\] for $x\ge \frac{n}{2}$, we have  $h_{n/2}(x)$ is strictly increasing on $[\tfrac{n}{2},\infty)$.
Hence $\rho(G)\le\rho(F_{a,b})\le\lambda_0<\eta(n)$.
 
 Now suppose that $a$ is odd.  There is one unmatched vertex in the maximum matching, and the partition into this vertex, the $a-1$ matched vertices, and the partite class of size $b$ is equitable.  Its quotient matrix is
\[
Q_{a,b}=
\begin{pmatrix}
		0 & 0 & b\\
		0 & 1 & b\\
		1 & a-1 & 0
\end{pmatrix}.
\]
By \Cref{equitable partition}, $\rho(F_{a,b})$ is the largest root of $p_{a,b}(x):=x^3-x^2-abx+b$.  We now distinguish the cases $a\ne n/2$ and $a=n/2$.
		
If $a\ne \frac{n}{2}$, then we claim that $\rho(F_{a,b})<\eta(n)$. Using the monotonicity of $p_{a,b}(x)$, it suffices to prove that $\eta(n)>\tfrac{n}{2}$, $p_{a,b}(\eta(n))>0$, and $p_{a,b}(x)$ is strictly increasing on $[\tfrac{n}{2},\infty)$. We compare this polynomial with $h_{n/2}$. Since $a+b=n$, we have $ab=b(n-b)=(\tfrac{n}{2})^2-(\tfrac{n}{2}-b)^2$.  Hence, using $h_{n/2}(\eta(n))=0$,
\[
p_{a,b}(\eta(n))
=((\tfrac{n}{2})^2-ab)\eta(n)+b-\frac{n}{2}
=(\frac{n}{2}-b)\bigl((\frac{n}{2}-b)\eta(n)-1\bigr).
\]
 Since $h_{n/2}(\tfrac{n}{2})=\tfrac{n}{2}-(\tfrac{n}{2})^2<0$, we have $\eta(n)>\tfrac{n}{2}$. Next, we show that $p_{a,b}(\eta(n))>0$.  Recall that $a\ne \tfrac{n}{2}$ and $b\ne \tfrac{n}{2}$.  Since $\tfrac{n}{2},a,b$ are all odd, $\tfrac{n}{2}-b$ is a nonzero even integer.  Thus, for $\tfrac{n}{2}-b>0$, we have $\tfrac{n}{2}-b\ge2$ and $(\tfrac{n}{2}-b)\eta(n)-1>0$.  For $\tfrac{n}{2}-b<0$, the two factors $\tfrac{n}{2}-b$ and $(\tfrac{n}{2}-b)\eta(n)-1$ are both negative.  In either case,
\[
p_{a,b}(\eta(n))>0.
\]
Finally, by $ab\le (\tfrac{n}{2})^2$,  we have
\[
p'_{a,b}(x)=3x^2-2x-ab\ge 3\left(\frac{n}{2}\right)^2-n-ab\ge n(\frac{n}{2}-1)>0
\]
for $x\ge \tfrac{n}{2}$.  Hence $p_{a,b}(x)$ is strictly increasing on $[\tfrac{n}{2},\infty)$.  If $\rho(F_{a,b})\ge \tfrac{n}{2}$, then $p_{a,b}(\rho(F_{a,b}))=0<p_{a,b}(\eta(n))$. We have $\rho(F_{a,b})<\eta(n)$ as  $p_{a,b}(x)$ is strictly increasing on $[\tfrac{n}{2},\infty)$. If $\rho(F_{a,b})<\tfrac{n}{2}$,  then we also have $\rho(F_{a,b})<\tfrac{n}{2}<\eta(n)$. Thus $\rho(F_{a,b})<\eta(n)$ for $a\ne \frac{n}{2}$.
		
If $a=b=\tfrac{n}{2}$, then \Cref{equitable partition} gives $\rho(F_{n/2,n/2})=\eta(n)$, and equality forces $G=F_{n/2,n/2}\in\mathcal{J}_{n,2}$.

\noindent\textbf{Case 3:} Let $n\equiv2\pmod4$ and let $k\ge4$ be even. Note that $\tfrac{n}{2}$ and $k-1$  are both odd.  We first identify the only possible
equality configuration. Let $G^\ast=G[V_0]\vee G[V_1]$.  By \Cref{join spectral bound},
\[
\rho(G)\le \rho(G^\ast) \le \frac{d_0+d_1+\sqrt{(d_0-d_1)^2+4|V_0||V_1|}}{2}.
\]
If $d_0+d_1\le k-2$, then
\[
\rho(G)\le\frac{k-2+\sqrt{(k-2)^2+4(\frac{n}{2})^2}}{2}<\psi(n,k).
\]
Thus equality forces $d_0+d_1=k-1$.  Since $k-1$ is odd, the inequality $|d_0-d_1|<k-1$ implies $|d_0-d_1|\le k-3$, and then
\[
\rho(G)\le\frac{k-1+\sqrt{(k-3)^2+4(\frac{n}{2})^2}}{2}<\psi(n,k).
\]
Hence equality forces $\{d_0,d_1\}=\{k-1,0\}$ by $d_0+d_1=k-1$ and $|d_0-d_1|=k-1$.

Relabel so that $d_0=k-1$ and $d_1=0$.  If $|V_0||V_1|\le (\frac{n}{2})^2-1$, then
\[
\rho(G)\le \frac{k-1+\sqrt{(k-1)^2+4((\frac{n}{2})^2-1)}}{2}=\psi(n,k).
\]
Equality in the preceding inequality requires $G$ to be $G[V_0]\vee G[V_1]$ with partite class sizes $\{\frac{n}{2}-1,\frac{n}{2}+1\}$, and equality in \Cref{join spectral bound}. Hence, the side with maximum degree $k-1$ is $(k-1)$-regular. Moreover, if $G[V_0]$ contains a member of $\mathcal{M}(\widehat{T})$, then the complete bipartite join would contain $\widehat{T}$, contradicting the hypothesis. This implies that  $G\in \mathcal{H}_{n,k-1}$.

It remains to rule out the balanced case $|V_0|=|V_1|=\frac{n}{2}$. We show that then $\rho(G)<\psi(n,k)$. Let $H:=G[V_0]$.	Since $d_1=0$, we have	$G^\ast=H\vee \overline{K_{n/2}}.$ Because $G\subseteq G^\ast$, it is enough to prove that
\[
\rho(G^\ast)<\psi(n,k).
\]
Let $\lambda:=\rho(G^\ast)$. Since all vertices in the independent side have the same neighbors, their Perron-vector coordinates are equal. Thus let $\mathbf{x^\ast} = (x_u^{\ast})_{u \in V(G^\ast)}$ be a Perron vector of $G^\ast$, defined by
\[
x_u^{\ast} = \begin{cases}
z_v, & \text{if } u = v \in V(H), \\
y,   & \text{if } u \in V(G^\ast)\setminus V(H).
\end{cases}
\] 
Let $r_v:=k-1-d_H(v)$ for $v\in V(H)$. Since both $k-1$ and $\frac{n}{2}$ are odd, $H$ cannot be $(k-1)$-regular, and therefore some $r_v$ is positive. This implies that $\sum_{v\in V(H)}r_v\ge 1$. Choose a vertex $u\in V(H)$ such that $z_u$ is minimum among all vertices in  the subgraph $H$.

Summing the eigen-equations over $V(H)$ gives
\[
\lambda \sum_{v\in V(H)}z_v=\sum_{v\in V(H)}d_H(v)z_v+(\frac{n}{2})^2y=(k-1)\sum_{v\in V(H)}z_v-\sum_{v\in V(H)}r_v z_v+(\frac{n}{2})^2y.
\]
Since every vertex on the independent side is adjacent to every vertex of $V(H)$, we also have
\[
\lambda y=\sum_{v\in V(H)}z_v.
\]
Hence
\begin{equation}\label{deficiency-identity}
\lambda^2-(k-1)\lambda-(\frac{n}{2})^2=-\frac{\lambda\sum_{v\in V(H)}r_vz_v}{\sum_{v\in V(H)}z_v}.
\end{equation}
The next claim shows that the right-hand side of \eqref{deficiency-identity} is greater than one.
\begin{claim}\label{bound}
 $\frac{\lambda\sum_{v\in V(H)}r_vz_v}{\sum_{v\in V(H)}z_v}> 1$.
\end{claim}
\begin{proof}[Proof of Claim \ref{bound}]
At the vertex $u$, we have
\[
\lambda z_u=\sum_{w\in N_H(u)}z_w+\frac{n}{2}y\ge d_H(u)z_u+\frac{n}{2}y=(k-1-r_u)z_u+\frac{n}{2}y.
\]
Thus
\[
(\lambda-k+1+r_u)z_u\ge \frac{n}{2}y=\frac{n\sum_{v\in V(H)}z_v}{2\lambda}
\]
and
\[
\frac{\lambda z_u}{\sum_{v\in V(H)}z_v}\ge \frac{n}{2(\lambda-k+1+r_u)}.
\]
Since $\sum_{v\in V(H)}r_vz_v\ge\left(\sum_{v\in V(H)}r_v\right)z_u$ by the choice of  $u$, we obtain
\begin{equation}\label{deficiency-lower}
\frac{\lambda\sum_{v\in V(H)}r_vz_v}{\sum_{v\in V(H)}z_v}\ge \frac{\lambda z_u \sum_{v\in V(H)}r_v}{\sum_{v\in V(H)}z_v}\ge \frac{n\sum_{v\in V(H)}r_v}{2(\lambda-k+1+r_u)}.
\end{equation}
Applying \Cref{join spectral bound} directly to $G^\ast=H\vee \overline{K_{n/2}}$,
we have
\[
\lambda\le \frac{k-1+\sqrt{(k-1)^2+4(\frac{n}{2})^2}}{2}<\frac{n}{2}+k-2,
\]
where the last inequality holds as $k\ge 4$ and $\tfrac{n}{2}\ge 3$. Hence
\[
\lambda-k+1+r_u<\frac{n}{2}+r_u-1.
\]
    
Recall that $\sum_{v\in V(H)}r_v\ge 1$.  If $r_u=0$, then
\[
\frac{\lambda\sum_{v\in V(H)}r_vz_v}{\sum_{v\in V(H)}z_v}\ge \frac{n\sum_{v\in V(H)}r_v}{2(\lambda-k+1+r_u)} > \frac{n/2}{n/2-1}> 1.
\]
 If $r_u\ge1$, then $\sum_{v\in V(H)}r_v\ge r_u$, and 
 \[
\frac{\lambda\sum_{v\in V(H)}r_vz_v}{\sum_{v\in V(H)}z_v}>\frac{r_u n/2}{n/2+r_u-1}\ge 1,
\]
 where the last inequality follows from $(r_u-1)(\tfrac{n}{2}-1)\ge 0$. Hence
\[
 \frac{\lambda\sum_{v\in V(H)}r_vz_v}{\sum_{v\in V(H)}z_v}> 1.
 \]
\end{proof}

Combining this with \eqref{deficiency-identity}, we obtain $\lambda^2-(k-1)\lambda-(\frac{n}{2})^2<-1$.	Therefore
\[
\lambda<\frac{k-1+\sqrt{(k-1)^2+4(\frac{n}{2})^2-4}}{2}=\psi(n,k),
\]
which proves that the balanced case cannot yield equality. Thus Case 3 is completed.
    
We finally verify that the three spectral alternatives in the definition of $\tau(n,k)$ are all attained for infinitely many admissible pairs $(n,k)$.

First let $k=2$ and $n\equiv2\pmod4$.  Write $n=4L+2$, and choose $L$ large enough that $n$ is sufficiently large.  Let $G'$ be obtained from $K_{n/2,n/2}$ by embedding a maximum matching in one partite set.  The embedded matching $F$ is triangle-free, and every component has order $2<|T|$. Hence $F$ contains neither $T$ nor any member of $\widehat{\cS(T)}$ in  \Cref{decomp sus tree}. In particular, $F$ is $\mathcal{M}(\widehat{T})$-free. Hence  \Cref{prop} shows that $G'$ is $\wh T$-free.  The partition into the unmatched vertex, the $\tfrac{n}{2}-1$ matched vertices in the same part, and the other part of size $\tfrac{n}{2}$ is equitable with quotient matrix
\[
\begin{pmatrix}
0&0&\frac{n}{2}\\
0&1&\frac{n}{2}\\
1&\frac{n}{2}-1&0
\end{pmatrix}.
\]
Its characteristic polynomial is
\[
       x^3-x^2-\frac{n^2}{4}x+\frac n2.
\]
Therefore, the largest eigenvalue is $\eta(n)$ by definition, and $\rho(G')=\eta(n)$.  This proves sharpness for infinitely many $n\equiv2\pmod4$when $k=2$.

Next consider the case in which $k\ge4$ is even and $n\equiv2\pmod4$.  Fix such an even $k$ and choose $L$ sufficiently large so that $n=4(k-1)L+2$ is sufficiently large.  Put
\[
a=\frac n2-1=2(k-1)L,\qquad b=\frac n2+1.
\]
Since $2(k-1)\mid a$,  we take $F$ to be the disjoint union of $L$ copies of $K_{k-1,k-1}$ and embed $F$ in the partite class of size $a$ of $K_{a,b}$. Since $F$ is $\mathcal{M}(\widehat{T})$-free, the resulting graph
\[
 G'=F\vee \overline{K_b}
\]
is $\wh T$-free by \Cref{prop} and has an equitable quotient matrix
\[
 \begin{pmatrix}
k-1 & b\\
  a & 0
\end{pmatrix}.
\]
Consequently,
\[
\begin{aligned}
 \rho(G')
 &=\frac{k-1+\sqrt{(k-1)^2+4ab}}{2}  \\
&=\frac{k-1+\sqrt{(k-1)^2+n^2-4}}{2}=\psi(n,k),
\end{aligned}
\]
because $4ab=4(\frac n2-1)(\frac n2+1)=n^2-4$.  Hence $\psi(n,k)$ is sharp for infinitely many  $n$ for each fixed even $k\ge4$.

Finally consider the  case in which $ n\not\equiv2\pmod4$ or $ k$  is odd. Fix $k\ge2$ and choose $L$ sufficiently large so that $n=4(k-1)L$ is sufficiently large.  Then $n\equiv0\pmod4$, and hence this pair $(n,k)$ belongs to this case.  Put
\[
        a=b=\frac n2=2(k-1)L.
\]
Since $2(k-1)\mid a$, we may take $F$ to be the disjoint union of $L$ copies of $K_{k-1,k-1}$.  Then $F$ is a $(k-1)$-regular triangle-free graph. Each component has order $2k-2<|T|$. Thus $F$ is $\cM(\wh T)$-free by the argument above.  The join
\[
        G'=F\vee \overline{K_{n/2}}
\]
is therefore $\wh T$-free.  Its natural partition into the two partite classes is equitable with quotient matrix
\[
        \begin{pmatrix}
        k-1 & n/2\\
        n/2 & 0
        \end{pmatrix}.
\]
Thus
\[
        \rho(G')=\frac{k-1+\sqrt{(k-1)^2+n^2}}{2}=\phi(n,k).
\]
Since there are infinitely many choices of $L$, the bound $\phi(n,k)$ is sharp for infinitely many values of $n$  for each fixed $k\ge 2$ in the case that $ n\not\equiv2\pmod4$ or $ k$  is odd.

The proof of \Cref{main} is complete.
\end{proof}

\begin{proof}[Proof of \Cref{disjoint}]
First let $k=8s+1$ with $s\ge1$, and let $T$ be a balanced tree on $2k+1$ vertices that is not a path and satisfies the Erd\H{o}s--S\'os hypothesis for all of its  nontrivial subtrees.  Choose $L$ sufficiently large so that $n=4(k-1)L$ is sufficiently large.  Then $n\equiv0\pmod4$ and $\tfrac n2$ is divisible by $2(k-1)$.
	
Let $F$ be the disjoint union of $\frac{n}{4(k-1)}$ copies of $K_{k-1,k-1}$.  Then $F$ is $\cM(\wh T)$-free: every component has order $2k-2<|T|$, and $F$ is triangle-free with maximum degree $k-1$.  Hence the graph $G'$  obtained from $K_{n/2,n/2}$ by embedding $F$ into one partite class belongs to $\mathcal{G}_{\tfrac n2,k-1}$ and is $\wh T$-free by\Cref{prop}.  Its natural equitable quotient matrix is
\[
\begin{pmatrix} 
k-1&\frac n2\\[2pt] \frac n2&0
\end{pmatrix}.
\]
Hence $\rho(G')=\phi(n,k)$.  Combining with \Cref{main},  this implies that $\spex(n,\widehat{T})=  \phi(n,k)$ and $\SPEX(n,\wh T)\subseteq\mathcal{G}_{\tfrac n2,k-1}$. Therefore, every spectral extremal graph has
\[
e_{\rm sp}=\left(\frac n2\right)^2+\frac{n(k-1)}4
\]
edges.
	
Moreover,
\[
\left\lfloor\frac{2n+k}{4}\right\rfloor=\frac n2+\left\lfloor\frac k4\right\rfloor=2(k-1)L+2s
\]
is even.  Thus the second equality condition in \Cref{classical} gives
\[
\ex(n,\wh T)
=\max_{n_0+n_1=n}\left\{n_0n_1+\left\lfloor\frac{n_0(k-1)}2\right\rfloor\right\}.
\]
Taking $n_0=\tfrac n2+2s$ and $n_1=\tfrac n2-2s$ yields
\[
\ex(n,\wh T)\ge\left(\frac n2+2s\right)\left(\frac n2-2s\right)+\frac{(k-1)(\frac n2+2s)}2=e_{\rm sp}+4s^2.
\]
Hence $\EX(n,\wh T)\cap\SPEX(n,\wh T)=\emptyset$ in the case $|T|=2k+1$.
	
Now let $k=8s$ with $s\ge1$, and let $T$ be a balanced tree on $2k$ vertices that is not a path, contains a matching of size $k$, and satisfies the Erd\H{o}s--S\'os hypothesis for all of its nontrivial subtrees.  Choose $L$ sufficiently large so that $n=2kL$ is sufficiently large.  Then $n\equiv0\pmod4$ and $\tfrac n2=kL$.
	
By K\"onig's theorem, $T$ has no vertex cover of size less than $k$.  Thus every vertex  covering $S$ has $|S|\ge k$, this implies that   every graph in $\widehat{\cS(T)}$ has maximum degree $k$ in \Cref{decomp sus tree}. Hence, in view of \Cref{decomp sus tree}, a graph with maximum degree at most
$k-1$ and with all components of order less than $2k$ contains no member of $\cM(\wh T)$.  Let $F$ be the disjoint union of $L$ copies of $K_k$.  Then $F$ is $(k-1)$-regular of  order $\frac n2$ and $\cM(\wh T)$-free. Embedding $F$ into one partite class of $K_{n/2,n/2}$ gives a graph $G' \in \mathcal{G}_{\tfrac n2,k-1}$ with spectral radius $\phi(n,k)$.  Combining with \Cref{main}, this implies that $\spex(n,\widehat{T})=  \phi(n,k)$ and $\SPEX(n,\wh T)\subseteq\mathcal{G}_{\tfrac n2,k-1}$. Thus every graph in $\SPEX(n,\wh T)$ belongs to $\mathcal{G}_{\tfrac n2,k-1}$ and has
\[
e_{\rm sp}=\left(\frac n2\right)^2+\frac{n(k-1)}4
\]
edges.
	
Since $T$ contains a matching of size $k$, the third equality condition in \Cref{classical} gives 
 \[
\ex(n,\wh T)=\max_{n_0+n_1=n}\left\{n_0n_1+\left\lfloor\frac{n_0(k-1)}2\right\rfloor\right\}.
\]
 Taking $n_0=\tfrac n2+2s$ and $n_1=\tfrac n2-2s$ gives
\[
\ex(n,\wh T)\ge\left(\frac n2+2s\right)\left(\frac n2-2s\right)+\frac{(k-1)(\frac n2+2s)}2=e_{\rm sp}+4s^2-s.
\]
Again, every ordinary extremal graph has strictly more edges than every spectral extremal graph. This implies that $\EX(n,\wh T)\cap\SPEX(n,\wh T)=\emptyset$.
	
Since $s$ is an arbitrary integer and the parameter $L$ can be chosen in infinitely many ways for each fixed $s$,  the first  construction proves the assertion for infinitely many $k=8s+1$, and the second proves it for infinitely many $k=8s$.   This completes the proof of \Cref{disjoint}.
\end{proof}
\section{Concluding remarks}
In the case that $n\not\equiv2\pmod4$ or $k$ is odd, we define the family $\mathcal{G}_{a,k-1}$ as the family of graphs obtained from the complete bipartite graph with partite classes of sizes $a$ and $n-a$ by embedding a $(k-1)$-regular $\mathcal{M}(\widehat{T})$-free graph into the partite class of size $a$.  Moreover, we show that if equality holds in \Cref{main}, then $\SPEX(n,\widehat{T})\subseteq \mathcal{G}_{\floor{n/2},k-1}\cup\mathcal{G}_{\ceil{n/2},k-1}$ for sufficiently large $n$ when $n\not\equiv2\pmod4$ or $k$ is odd. In particular, we observe that $\mathcal{G}_{\ceil{n/2},k-1}=\emptyset$ for $n\equiv1\pmod4$ and even $k$, because $\ceil{n/2}(k-1)$ is odd. Similarly, $\mathcal{G}_{\floor{n/2},k-1}=\emptyset$ for $n\equiv3\pmod4$ and even $k$. Consequently, if equality holds in \Cref{main}, then $\SPEX(n,\widehat{T})\subseteq \mathcal{G}_{\floor{n/2},k-1}$ when $n\equiv1\pmod4$ and $k$ is even, and $\SPEX(n,\widehat{T})\subseteq \mathcal{G}_{\ceil{n/2},k-1}$ when $n\equiv3\pmod4$ and $k$ is even.

Since a path is a special balanced tree, our argument yields  the spectral Tur\'an number of $\widehat{P}_\ell$ for sufficiently large $n$ as a corollary, which was determined by Zhang~\cite{Z}. Moreover, together with the lower-bound constructions for $\widehat{P}_\ell$ in \cite{Z}, we also characterize all extremal graphs for $\spex(n,\widehat{P}_\ell)$ when $n$ is large enough. However, characterizing the extremal graphs for $\spex(n,\widehat{T})$ for a general tree $T$ remains difficult, because suitable lower-bound constructions are not known; this is an interesting open problem.

Finally, recall that \Cref{classical} is proved for $n\ge4(4k)^6$, while \Cref{main} only holds for sufficiently large $n$, as our proof relies on \Cref{stability}. It would be interesting to find an alternative approach that yields an explicit bound on $n$ in \Cref{main} without appealing to \Cref{stability}.
\section*{Acknowledgments}
The authors extend their sincere gratitude to their supervisor, Professor Xing Peng, for his insightful suggestions, which have greatly improved the presentation of this work. 


\begin{thebibliography}{99}
\bibitem{CDT22}S. Cioab\u{a}, D. Desai and M. Tait, The spectral radius of graphs with no odd wheels, {\it European J. Combin.}, {\bf 99} (2022),  No.103420, 19pp.
\bibitem{CFTZ20}S. Cioab\u{a}, L. Feng, M. Tait and X.-D. Zhang, The maximum spectral radius of graphs without friendship subgraphs, {\it Electron. J. Combin.}, {\bf 27}(4) (2020), No.4.22, 19pp.
\bibitem{D22}D. Desai, L. Kang, Y. Li, Z. Ni, M. Tait and J. Wang, Spectral extremal graphs for intersecting cliques, \emph{Linear Algebra Appl.}, \textbf{644} (2022), 234-258.
\bibitem{FTZ24}L. Fang, M. Tait and M. Zhai, Tur\'an numbers for non-bipartite graphs and applications to spectral extremal problems, \emph{Discrete Math.} \textbf{348}(10) (2025), No. 114527, 16pp.
 \bibitem{GR01}C. Godsil and G. Royle, \emph{Algebraic Graph Theory}, Grad. Texts in Math., vol. 207, Springer-Verlag, New York, 2001.
\bibitem{LP22}Y. Li and Y. Peng, The spectral radius of graphs with no intersecting odd cycles, \emph{Discrete Math.}, \textbf{345}(8) (2022), No. 112907, 16pp.
\bibitem{M05}A. McLennan, The Erd\H{o}s-S\'os conjecture for trees of diameter four. \emph{J. Graph Theory}, \textbf{49}(4) (2005), 291-301.
\bibitem{NWK23} Z. Ni, J. Wang and L. Kang, Spectral extremal graphs for disjoint cliques, \emph{Electron. J. Combin.}, \textbf{30}(1) (2023), No.1.20, 16pp.
\bibitem{N02}V. Nikiforov, Some inequalities for the largest eigenvalue of a graph, \emph{Combin. Probab. Comput.}, \textbf{11}(2) (2002),  179–189.
\bibitem{N10}V. Nikiforov, The spectral radius of graphs without paths and cycles of specified length, \emph{Linear Algebra Appl.}, \textbf{432}(9) (2010), 2243–2256.
\bibitem{S89} A. Sidorenko, Asymptotic solution for a new class of forbidden r-graphs. \emph{Combinatorica}, \textbf{9}(2) (1989), 207-215.
\bibitem{S} M. Simonovits, Extremal graph problems with symmetrical extremal graphs. Additional chromatic conditions, \emph{Discrete Math.}, \textbf{7} (1974), 349-376.
\bibitem{MT19}M. Tait, The Colin de Verdi\`ere parameter, excluded minors, and the spectral radius, \emph{J. Combin. Theory Ser. A}, \textbf{166} (2019), 42–58.
\bibitem{T41} P. Tur\'an, On an extremal problem in graph theory, \emph{Mat. Fiz. Lapok}, \textbf{48} (1941), 436--452.
\bibitem{WNKF}J. Wang, Z. Ni, L. Kang and Y.-Z.~Fan, Spectral extremal graphs for edge blow-up of star forests, {\it Discrete Math.}, {\bf 347} (2024), 114141.
\bibitem{WKX23}J. Wang, L. Kang and Y. Xue, On a conjecture of spectral extremal problems, {\it J. Combin. Theory Ser. B}, {\bf 159} (2023), 20--41.
\bibitem{Y21} L.-T. Yuan, Extremal graphs for odd wheels, \emph{J. Graph Theory}, \textbf{98}(4) (2021), 691--707.
\bibitem{Zhan}X. Zhan, {\it Matrix Theory}, Graduate Studies in Mathematics, vol. 147, Amer. Math. Soc., Providence, RI, 2013.
\bibitem{Z}W. Zhang, Spectral extrema of graphs forbidding a fan, arXiv:2508.05911.
\bibitem{Zhu}X. Zhu, X. Wang, Y. Zhang and F. Zhang, Tur\'an problems for suspension of a balanced tree, \emph{J. Graph Theory}, \textbf{113}(1) (2026) 61--72.
\end{thebibliography}
\end{document}